\documentclass[11pt]{article}
\usepackage{amsmath,amsthm,amsfonts,amssymb,url,graphicx,xcolor,authblk,longtable}
\newtheorem{theorem}{Theorem}[section]
\newtheorem{lemma}[theorem]{Lemma}
\newtheorem{corollary}[theorem]{Corollary}
\newtheorem{remark}{Remark}

\newtheorem{example}{Example}[section]

\newcommand{\B}{\ensuremath{\mathcal{B}}}

\newcommand{\fb}[1]{\makebox(40,16){#1}} 
\newcommand{\fbb}[1]{\framebox(40,16){#1}} 
\newcommand{\fbbl}[1]{\framebox(80,16){#1}} 
\newcommand{\third}{\ensuremath{\mathsf{third}}} 
\newcommand{\STS}{\ensuremath{\mathsf{STS}}}

\title{Block-avoiding point sequencings of arbitrary length in Steiner triple systems}
\author[1]{Douglas R.\ Stinson%
\thanks{D.R.\ Stinson's research is supported by  NSERC discovery grant RGPIN-03882.}}
\author[2]{Shannon Veitch}
\affil[1]{David R.\ Cheriton School of Computer Science, University of Waterloo,
Waterloo, Ontario, N2L 3G1, Canada}
\affil[2]{Department of Combinatorics and Optimization, University of Waterloo,
Waterloo, Ontario, N2L 3G1, Canada}

\begin{document}
\maketitle

\begin{abstract}
An \emph{$\ell$-good sequencing} of an \STS$(v)$ is a permutation of the points of the design such that no $\ell$ consecutive points in this permutation contain a block of the design. We prove that, for every integer $\ell \geq 3$, there is an $\ell$-good sequencing of any \STS$(v)$ provided that $v$ is sufficiently large. We also prove some new nonexistence results for $\ell$-good sequencings of  \STS$(v)$.
\end{abstract}

\section{Introduction}
\label{intro.sec}

A \emph{Steiner triple system of order $v$} is a pair $(X, \B)$, where 
$X$ is a set of $v$ \emph{points} and $\B$ is a set of 3-subsets of $X$ (called 
\emph{blocks}), such that every pair of points occur in exactly one block.
We will abbreviate the phrase ``Steiner triple system of order $v$'' to 
\STS$(v)$. It is well-known that an \STS$(v)$ contains exactly $v(v-1)/6$ blocks, and an 
\STS$(v)$ exists if and only  if $v \equiv 1,3 \bmod 6$. 
The definitive reference for Steiner triple systems is the book \cite{CR} by Colbourn and Rosa.

The following problem was introduced by Kreher and Stinson in \cite{KS2}.
Suppose $(X, \B)$ is an \STS$(v)$ 
 and let $\ell \geq 3$ be an integer. 
An \emph{$\ell$-good sequencing} of $(X, \B)$ is a permutation
$\pi = [x_1\; x_2 \;  \cdots \;  x_v]$ of $X$ 
such that no $\ell$ consecutive points in the permutation contain a block in $\B$.
(Some related but different sequencing problems for \STS$(v)$ are studied in \cite{AKP} and \cite{KS}.)

\begin{remark}
\label{rem1}
{\rm 
We observe that an $\ell$-good sequencing is automatically an $m$-good sequencing if $m < \ell$.}
\end{remark}

It is an interesting question if there exists, for a given integer $\ell \geq 3$, an 
$\ell$-good sequencing of a specified \STS$(v)$, or if there exists an $\ell$-good sequencing of all \STS$(v)$ 
(for sufficiently large values of $v$).  The following  results were proven in \cite{KS2}:
\begin{itemize}
\item any \STS$(v)$ with $v > 3$ has a $3$-good sequencing,
\item any \STS$(v)$ with $v > 71$ has a $4$-good sequencing,
\item the (unique) \STS$(7)$ and \STS$(9)$ do not have a $4$-good sequencing, and 
\item all \STS$(13)$ and \STS$(15)$ have a $4$-good sequencing.
\end{itemize}

It was  conjectured in  \cite{KS2}, for any integer $\ell \geq 3$, that  there exists an 
integer $n(\ell)$ such that any \STS$(v)$ with $v > n(\ell)$ has an $\ell$-good sequencing. 
We prove this conjecture in  Section \ref{exist.sec} of this paper and we show that $n(\ell) \in O(\ell^6)$.
We also prove a nonexistence result, in Section \ref{counting.sec}, namely, that an \STS$(v)$ with $v > 7$ 
cannot have an $\ell$-good sequencing if $\ell \geq (v + 2)/3$.

We will use the following notation in the remainder of this paper. Suppose $(X,\B)$ is an \STS$(v)$.
Then, for any pair of points $x,y$, 
let $\third(x,y) = z$ if and only if $\{x,y,z\} \in \B$.
The function $\third$ is well-defined because every pair of points occurs in a 
unique block in $\B$.

\section{A counting argument}
\label{counting.sec}

In this section, we generalize a counting argument from \cite[\S 3.1]{KS2} that was used to prove the nonexistence of $4$-good sequencings of \STS$(7)$ and \STS$(9)$.
Let $v \geq 7$ and $\ell \geq 3$ be integers. Suppose we take the points of an $\STS(v)$ to be $1, \dots, v$. Without loss of generality, suppose, by relabelling points if necessary, that $[1 \; 2 \; 3 \; \cdots \; v]$ is an $\ell$-good sequencing of an $\STS(v)$. We say that a block $B$ is of \textit{type i} if $| B \cap \{1,2, \dots, \ell\}| = i$. Clearly, we must have $i \in \{0,1,2\}$.

For $i = 0,1,2$, let $b_i$ denote the number of blocks of type $i$. Since the sequencing is $\ell$-good, we know that $b_2 = \binom{\ell}{2}$. Since each point appears in $(v-1)/2$ blocks, we have
\[b_1 = \ell \left( \frac{v-1}{2} - (\ell-1) \right).\]
Finally, because the total number of blocks is $v(v-1)/6$, we have
\begin{eqnarray*}
b_0 &=& \frac{v(v-1)}{6} - \ell \left( \frac{v-1}{2} - (\ell-1) \right) - \binom{\ell}{2} \\
&=& \frac{v(v-1)}{6} - \frac{\ell (v - \ell)}{2}.
\end{eqnarray*}

Consider a block of type 0, say $B = \{x,y,z\}$ where $x < y < z$. We must have $x \leq v - \ell$ because otherwise $B \subseteq \{v-\ell+1, \dots, v-2, v-1, v\}$. Since $B$ is of type 0, we also have that $x \geq \ell + 1$. For each such $x$ such that $\ell + 1 \leq x \leq v - \ell$,
 we have $z\in \{ x + \ell, \dots, v-1, v\}$, so there are $v-(x+\ell-1)$ possible values for $z$. 
It follows that there can be at most
\[
\sum_{x = \ell + 1}^{v-\ell} (v-(x+\ell-1)) 
= \frac{(v-2\ell)(v-2\ell + 1)}{2}
\]
blocks of type $0$.
Since there are $b_0 = v(v-1)/6 - \ell (v - \ell)/2$ blocks of type $0$, we obtain
\[ \frac{v(v-1)}{6} - \frac{\ell (v - \ell)}{2} \leq \frac{(v-2\ell)(v-2\ell + 1)}{2},\]
which simplifies to give \[ 0 \leq (3\ell-2v)(3\ell - v - 2).\]
We are assuming $v \geq 7$, so $(v+2)/3 + 1 < 2v/3$. Hence, $\ell \leq (v+2)/3$ or $\ell \geq 2v/3$. 
Therefore there does not exist a $(\lfloor (v+2)/3\rfloor + 1)$-good sequencing of an \STS$(v)$. 
Then, it follows from Remark \ref{rem1} that we cannot have an $\ell$-good sequencing with $\ell \geq 2v/3$.

Summarizing  the above discussion, we have the following theorem.

\begin{theorem}
\label{nonexist.thm}
If an $\STS(v)$ with $v \geq 7$ has an $\ell$-good sequencing,  then
$
\ell \leq (v+2)/3.
$
\end{theorem}

By analyzing the case of equality in Theorem \ref{nonexist.thm} more carefully, we can rule out the existence of
an $\ell$-good sequencing of an $\STS(3\ell-2)$ whenever $\ell > 3$ is odd (note that $\ell$ must be odd for an $\STS(3\ell-2)$ to exist). 

\begin{theorem}
\label{nonexist2.thm}
If $\ell > 3$ is an odd integer, then no $\STS(3\ell-2)$ has an $\ell$-good sequencing.
\end{theorem}

\begin{proof}
Suppose, by way of contradiction, that there is an $\ell$-good sequencing of an 
$\STS(3\ell-2)$ for some $\ell > 3$.
From the proof of Theorem \ref{nonexist.thm}, 
there are $v-2\ell = \ell-2$ blocks of type $0$ that contain the point $\ell+1$. 
Within these  $\ell-2$ blocks, the point $\ell+1$ occurs with $2\ell-4$ other points in the set 
$\{ \ell+2, \dots , v\}$, which has cardinality $2\ell-3$. It follows that 
the point $\ell+1$ must occur in exactly one block of type $1$.

Since every point occurs in exactly $(v-1)/2$ blocks, the point $\ell+1$ must occur in
\[  \frac{v-1}{2} - (\ell-2) - 1 = \frac{\ell-1}{2}\] blocks of type $2$.
We have assumed $\ell > 3$, so the point $\ell+1$ must occur in at least two blocks of type $2$.
However, if the point $\ell+1$ occurs in a block $B$ of type $2$, then $1 \in B$ 
(otherwise, the sequencing is not
$\ell$-good). But the pair $\{ 1, \ell +1 \}$ is only contained in one block, so we have a contradiction.
\end{proof}

\begin{example}
\label{13.exam}
{\rm Consider an $\STS(13)$. Here, we have that
$(13 + 2)/3 = 5$. 
Theorem \ref{nonexist.thm} tells us that there is no 6-good sequencing of an $\STS(13)$,
and Theorem \ref{nonexist2.thm} extends this to show that no $\STS(13)$ has a 5-good sequencing. 
Similarly, because $(19 + 2)/3 =7$, there is no 7-good sequencing of an $\STS(19)$.}
\end{example}

\section{Existence of $\ell$-good sequencings}
\label{exist.sec}

For any integer $\ell \geq 3$, it was conjectured in \cite{KS2} that all
``sufficiently large'' \STS$(v)$ have $\ell$-good sequencings. 
The conjecture was proven for $\ell = 3$ and $\ell = 4$ in \cite{KS2}. 
Here, we prove the conjecture for all $\ell \geq 3$.

We use a greedy strategy 
similar to the algorithms discussed in \cite{KS2}. The idea is to successively choose 
$x_1, \dots , x_{v}$ in such a way that we end up with an $\ell$-good sequencing of a given \STS$(v)$.
However, this strategy is too simple to guarantee success, so we need to incorporate some 
modifications that we will discuss subsequently.

In general, when we choose a value for $x_i$, it must be distinct from 
$x_1, \dots , x_{i-1}$, of course. It is also required that
\begin{equation}
\label{third.eq}
 x_i \not\in P_{i,\ell} = \{ \third(x_{j},x_{k}) : i-\ell+1 \leq j < k \leq i-1\}.
\end{equation}
Note that $|P_{i,\ell}| \leq \binom{\ell-1}{2}$. For ease of notation in the rest of this section, we will define $L = \binom{\ell-1}{2}$.

There will be a permissible choice for $x_i$ provided that
$i-1 + L \leq v-1$, which is equivalent to the condition $i \leq v- L$.
Thus we can define $x_1, x_2, \dots , x_{v- L}$ in such a way that they satisfy the relevant 
conditions---this is what we term the ``greedy strategy.'' 
Our task is then to somehow fill in the last $L$ positions of
the sequencing, after appropriate modifications, to satisfy the desired properties.
We describe how to do this now, for sufficiently large values of $v$.

Suppose that $[x_1\; x_2\; \cdots \; x_{v-L}]$ is an \emph{$\ell$-good partial sequencing} of
$X = \{1, \dots ,v\}$ (that is, there is no block contained in any $\ell$ consecutive points in 
the sequence $[x_1\; x_2\; \cdots \; x_{v-L}]$). 
Let \[X \setminus \{ x_1, x_2, \dots , x_{v-L}\} = \{ \alpha_1, \dots , \alpha_L \}.\]
Suppose we temporarily define $x_{v-L+i} = \alpha_i$ for $1 \leq i \leq L$.

\subsection{Segments}

We will construct $L$ disjoint \emph{segments}, denoted $\mathcal{S}_{i}$, $1 \leq {i} \leq L$.
Each segment $\mathcal{S}_i$ will consist of 
\begin{itemize}
\item for $i \geq 2$, a \emph{left buffer}, $\mathcal{B}_{i}^L$ 
(however, we will not require a left buffer for the first segment),
\item a \emph{core} denoted by $\mathcal{C}_{i}$, 
\item a \emph{right buffer}, $\mathcal{B}_{i}^R$, and 
\item an \emph{overflow}, $\mathcal{O}_{i}$.
\end{itemize} 
The above are all \emph{ordered} lists of points in the \STS$(v)$. See Figure \ref{segment.fig}.

\begin{figure}
\begin{center}
\fb{$\mathcal{S}_i= $}\fbb{$\mathcal{B}_{i}^L$}\fbb{$\mathcal{C}_{i}$}\fbb{$\mathcal{B}_{i}^R$}\fbb{$\mathcal{O}_{i}$}
\end{center}
\caption{A segment  $\mathcal{S}_i$}
\label{segment.fig}
\end{figure}

Each buffer has size $\ell-1$ (except that the first left buffer has size $0$) 
and the size of the core will be denoted by $c_{i}$. We will discuss the value of $c_i$ and the 
size of the the overflow a bit later. 
The basic strategy of our algorithm will be to (if necessary) swap each $\alpha_i$ with either
\begin{description}
\item[(a)] one of $\alpha_{i+1}, \dots , \alpha_L$ (there are $L-i$ choices here), or
\item[(b)] a point  from the core $\mathcal{C}_i$ (there are $c_i$ choices for such a point).
\end{description} 
We will perform a sequence of swaps of this type, for $i = 1, 2, \dots ,  L$.

When we perform a swap $\alpha_i \leftrightarrow x_j \in \mathcal{C}_{i}$, 
we need to ensure that two conditions are satisfied:
 \begin{enumerate}
 \item $x_j \not\in P_{v-L+i,\ell}$ (from (\ref{third.eq})), and 
 \item $\alpha_i$ does not lead to the formation of  a new block among any $\ell$ consecutive points in 
 $\mathcal{S}_i$.
 \end{enumerate}

 \subsection{The core}
 
 First, we consider how big  the core $\mathcal{C}_i$ needs to be.
  When we are defining $x_{v-L+i}$, if we have $L+1$ choices, then one of them must be good (i.e., not in the set
$P_{v-L+i, \ell}$). 
 The number of choices in \textbf{(a)} or \textbf{(b)} 
 is $c_i + L-i + 1$, so we want $c_i + L-i + 1 \geq L+1$, or $c_i \geq i$, for $1 \leq i \leq  L$.
 (The  ``$+1$'' term on the left side of the inequality 
 accounts for the possibility that $\alpha_i$ might already be a good choice, 
 in which case no swap would be necessary.) Thus, from this point on, we will assume that $c_i = i$ for all $i$.
 
 \subsection{The overflow}
 
 Define $\mathcal{T}_i = 
 \mathcal{B}_i^L \cup \mathcal{C}_i \cup \mathcal{B}_i^R$. 
  We  need to ensure that there are no blocks contained in $\ell$ consecutive points of 
 $\mathcal{S}_i$ after a point 
 $\alpha_i$ is swapped for a point in $\mathcal{C}_i$. 
This is accomplished by considering 
 blocks containing two points in $\mathcal{T}_i$ and placing the relevant third points 
 ``out of harm's way''  in the overflow.

 For now, we assume that $i \geq 2$. 
 We only need to consider blocks contained in
 $\ell$ consecutive points in $\mathcal{T}_i$, because 
 \begin{itemize}
 \item
 the last point in the core and the first point in the overflow are not contained in $\ell$ consecutive points, 
 and 
\item
 for $i \geq 2$, the first point in the core and the last point in the previous overflow are not contained in $\ell$ consecutive points.
 \end{itemize} 
 
 Denote the points (in order) in $\mathcal{T}_i$ by $z_1, \dots , z_{i + 2\ell - 2}$.
Define $\mathcal{J}_i$ to consist of all the ordered pairs $(j_1,j_2)$ such that 
 \begin{itemize}
 \item $1 \leq j_1 < j_2 \leq i + 2\ell - 2$ and
 \item $j_2-j_1 \leq \ell-1$.
 \end{itemize}

\begin{lemma} 
\label{J.lemma}
For  $2 \leq i \leq L$, we have
$|\mathcal{J}_i| =(\ell-1)(i + (3\ell-4)/2)$. 
Also, $|\mathcal{J}_1| = \binom{\ell}{2}$.
\end{lemma}

\begin{proof}
First, assume $2 \leq i \leq L$. Let $1 \leq d \leq \ell-1$.  There are exactly $i + 2\ell-2 - d$ pairs 
$(j_1,j_2) \in \mathcal{J}_i$ with $j_2 - j_1 = d$. Hence,
\begin{eqnarray*}
|\mathcal{J}_i| & = & \sum_{d=1}^{\ell-1} (i + 2\ell-2 - d)\\
& = & (\ell-1) (i + 2(\ell-1)) - \frac{\ell(\ell-1)}{2} \\
& = & (\ell-1) \left(i + \frac{3\ell-4}{2} \right).
\end{eqnarray*}

Now let's look at the initial case,  $i=1$. Here, we have $\mathcal{T}_1 = 
\mathcal{C}_1 \cup \mathcal{B}_1^R$, so $|\mathcal{T}_1| = \ell$.
$\mathcal{J}_1$ will will consist of all $\binom{\ell}{2}$ ordered pairs
$(j_1,j_2)$ such that 
$1 \leq j_1 < j_2 \leq \ell$.
\end{proof}

Next, define \[Y = \big\{ \third(z_{j_1},z_{j_2}): (j_1,j_2) \in \mathcal{J}_i \big\} 
\setminus (\mathcal{T}_i \cup \mathcal{S}_1 \cup \cdots \cup \mathcal{S}_{i-1}).\]
Note that, when we define $Y$ we  omit any points $\third(z_{j_1},z_{j_2})$ that have 
already appeared in $\mathcal{T}_i \cup \mathcal{S}_1 \cup \cdots \cup \mathcal{S}_{i-1}$.
Denote the points in $Y$ as $y_1, \dots , y_m$. Clearly, $m \leq |\mathcal{J}_i|$.

Having already chosen the points in $\mathcal{T}_i$, we want to ``pre-specify'' the location of the $m$ 
points $y_1, \dots , y_m$ in the overflow $\mathcal{O}_{i}$. This is done according to the algorithm in Figure 
\ref{overflow.fig}.  We should explain the spacing of points $Y = \{y_1, \dots , y_m\}$ in the overflow. 
We want to avoid a situation where there could be three points (within $\ell$ consecutive points)
that might comprise a block.  The initial gap of length $\ell-2$ ensures that the last two points of 
$\mathcal{B}_i^R$ and $y_1$ are not contained in $\ell$ consecutive points. 
Also, the remaining gaps are large enough to guarantee that no three points $y_i$, $y_{i+1}$ and $y_{i+2}$ are  contained in $\ell$ consecutive points.

\begin{figure}[tb]
\begin{center}
\begin{tabular}{l}
{\bf Input}: the set $Y = \{y_1, \dots , y_m\}$ and an integer $i$, $1 \leq i \leq L$\\
Insert the points in $Y$ into $\mathcal{O}_{i}$ as follows:\\
{\bf if} $\ell$ is even {\bf then}\\
\quad\quad leave an initial gap of length $\ell-2$ and then insert $y_1$\\
\quad\quad {\bf for} $2 \leq i \leq m$ {\bf do}\\
\quad\quad\quad\quad leave a gap of length $(\ell-2)/2$ between $y_{i-1}$ and $y_i$\\
{\bf else} (i.e., $\ell$ is even) \\
\quad\quad leave an initial gap of length $\ell-2$ and then insert $y_1$\\
\quad\quad {\bf for} $2 \leq i \leq m$ {\bf do}\\
\quad\quad\quad\quad {\bf if} $i$ is even {\bf then} \\
\quad\quad\quad\quad\quad\quad leave a gap of length $(\ell-3)/2$ between $y_{i-1}$ and $y_i$\\
\quad\quad\quad\quad {\bf else} (i.e., $i$ is odd) \\
\quad\quad\quad\quad\quad\quad  leave a gap of length $(\ell-1)/2$ between $y_{i-1}$ and $y_i$
\end{tabular}
\end{center}
\caption{Pre-specifying elements in the overflow $\mathcal{O}_{i}$}
\label{overflow.fig}
\end{figure}

We can now compute the length of an overflow. 

\begin{lemma} 
\label{overflow.lem}
For any integer $i$ such that  $1 \leq i \leq L$, we have $|\mathcal{O}_i| \leq \ell(|\mathcal{J}_i| + 1)/2 - 1$.
\end{lemma}
\begin{proof}
First, suppose $\ell$ is even. Using the notation above, the overflow consists of $\ell - 2$ initial values followed by the values in $Y$, each separated by $(\ell - 2)/2$ points. Let $|Y| = m$. Then the overflow has length
\[
\frac{(m-1)(\ell-2)}{2} + m + \ell - 2 = \frac{\ell(m + 1)}{2} - 1.
\]
Since $m \leq |\mathcal{J}_i|$, it follows that
\[
|\mathcal{O}_i| \leq \frac{\ell(|\mathcal{J}_i| + 1)}{2} - 1.
\]

Now suppose $\ell$ is odd. Then the overflow consists of $\ell - 2$ initial values followed by the values in $Y$ separated by $(\ell - 3)/2$ or $(\ell - 1)/2$ points, alternating. Let $|Y| = m$. If $m$ is odd then the overflow has length
\[
\frac{(m-1)}{2} \times \frac{(\ell-3)}{2} + \frac{(m-1)}{2} \times \frac{(\ell - 1)}{2} + m + \ell - 2 = \frac{\ell(m + 1)}{2} - 1.
\]
Otherwise, $m$ is even so the overflow has length
\[
\frac{m}{2} \times \frac{(\ell-3)}{2} + \frac{(m-2)}{2} \times \frac{(\ell-1)}{2} + m + \ell - 2 = \frac{\ell(m + 1) - 3}{2}.
\]
Since 
\[
\frac{\ell(m + 1) - 3}{2} < \frac{\ell(m + 1)}{2} - 1,
\]
and $m \leq |\mathcal{J}_i|$, it follows that 
\[
|\mathcal{O}_i| \leq \frac{\ell(|\mathcal{J}_i| + 1)}{2} - 1
\]
for all $i \geq 1$.
\end{proof}


\begin{corollary}
\label{overflow.cor}
For any integer $i$ such that  $2 \leq i \leq L$, we have 
\[ |\mathcal{O}_i| \leq \frac{i(\ell^2 - \ell)}{2} + \frac{3 \ell^3 - 7 \ell^2 +6\ell - 4 }{4}.\] 
Also, \[|\mathcal{O}_1| \leq \frac{\ell^3 - \ell^2 +2\ell-4}{4}.\]
\end{corollary}

\begin{proof}
Applying Lemmas \ref{J.lemma} and \ref{overflow.lem}, we obtain
\begin{eqnarray*}
|\mathcal{O}_i| & \leq &\frac{\ell \left( (\ell-1) \left(i +  \frac{3\ell - 4}{2}\right) + 1\right)}{2} - 1\\
& = & \frac{i(\ell^2 - \ell)}{2} + \frac{3 \ell^3 - 7 \ell^2 +6\ell - 4 }{4}
\end{eqnarray*}
and
\begin{eqnarray*}
|\mathcal{O}_1| &\leq& \frac{ \ell \left( \binom{l}{2} + 1 \right)}{2} - 1\\
& = & \frac{\ell^3 - \ell^2 +2\ell-4}{4}.
\end{eqnarray*}

\end{proof}

\subsection{The gap}

After carrying out the operations described in Figure \ref{overflow.fig}, we fill in the rest of the overflow
$\mathcal{O}_{i}$ using what we call the ``modified greedy strategy.'' Each time we choose a new point $x_j$, 
we make sure that $x_j \not\in P_{j,\ell}$, as per (\ref{third.eq}). However, we additionally need to make sure that there is no block contained in a set of $\ell$ consecutive points that may include points $x_{j'}$ with $j' > j$ that have been predefined as a result of the algorithm in Figure \ref{overflow.fig}. In order to ensure that this can be done, we include a \emph{gap}, denoted $\mathcal{G}$, 
that follows the last overflow, $\mathcal{O}_{L}$. 
$\mathcal{G}$ will contain elements after $\mathcal{O}_{L}$, up to, but not including, the last $L$ points in the sequencing. The gap will be filled using the greedy strategy. See Figure \ref{gap.fig}.

\begin{figure}
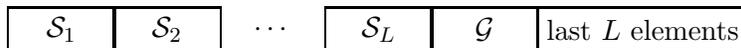

\begin{center}
\fbb{$\mathcal{S}_{1}$}\fbb{$\mathcal{S}_{2}$}\fb{$\cdots$}\fbb{$\mathcal{S}_{L}$}\fbb{$\mathcal{G}$}\fbbl{last $L$ elements}
\end{center}
\caption{The overall structure of the sequencing}
\label{gap.fig}
\end{figure}

Let's determine how big the gap needs to be. First, consider 
the second last element of $\mathcal{O}_{L}$. The last element of $\mathcal{O}_{L}$, 
say $x_{\kappa}$ has been pre-specified to be the value $y_m$. Now, as we have already mentioned, 
$x_{\kappa-1} \not\in P_{\kappa-1,\ell}$, which rules out no more than $L$ values for $x_{\kappa-1}$.
Also, \[ x_{\kappa-1} \not\in \{ \third(x_{j},x_{\kappa}) : \kappa-\ell+1 \leq j \leq \kappa -2\}.\]
This rules out up to $\ell-2$ additional values for $x_{\kappa-1}$.
The number of unused values is $|\mathcal{G}| + L + 1$, since we have not yet defined 
$x_{\kappa-1}$, any element in the gap, or any of the last $L$ elements. So we require
$ L + \ell-2 +1 \leq |\mathcal{G}| + L +1$, or $|\mathcal{G}| \geq \ell-2$, in order to ensure that
$x_{\kappa-1}$ can be defined. 

We should also consider the element immediately preceding 
$y_{m-1} = x_{\kappa}$. Following $x_{\kappa}$, there is are $\beta$ undefined elements, followed by
$y_m$, where 
\[ \beta \in \left\{ \frac{\ell-1}{2}, \frac{\ell-2}{2}, \frac{\ell-3}{2} \right\}.\]
Suppose we have defined all elements up to but not including $x_{\kappa-1}$. Also, the values $x_{\kappa}$
and $x_{\kappa+\beta + 1}$ have been prespecified.

The restrictions on $x_{\kappa-1}$ are as follows:
\begin{itemize}
\item $x_{\kappa-1} \not\in P_{\kappa-1,\ell}$ (as before, which rules out at most $L$ values),
\item $x_{\kappa-1} \not\in \{ \third(x_{j},x_{\kappa}) : \kappa-\ell+1 \leq j \leq \kappa -2\}$ (as before, which rules out at most $\ell-2$ values), 
\item $x_{\kappa-1} \neq \third(x_{\kappa},x_{\kappa + \beta + 1})$ (at most one value is ruled out here)
\item $x_{\kappa-1} \not\in \{ \third(x_{j},x_{\kappa + \beta+1}) : \kappa +\beta -\ell+2 \leq j \leq \kappa -2\}$
(at most $\ell - \beta - 3$ values are ruled out).
\end{itemize}
Therefore the total number of values that are ruled out is at most
\[ L + \ell-2 + 1 + \ell - \beta - 3 = L + 2\ell - \beta - 4.\]
Since the $\beta$ elements between
$x_{\kappa}$ and $x_{\kappa+\beta + 1}$ have not yet been defined, the number of available elements is 
$|\mathcal{G}| + L + \beta + 1$. 
Therefore we can choose a value for $x_{\kappa-1}$  provided that
\[ |\mathcal{G}| + L + \beta + 1 \geq L + 2\ell - \beta - 4 + 1,\]
which simplifies to give
\[ |\mathcal{G}| \geq 2 (\ell - \beta - 2) .\]

If $\ell$ is even, then $\beta = (\ell-2)/2$
and it suffices to take
\[ |\mathcal{G}| \geq 2 \left( \ell - \frac{\ell-2}{2} - 2\right) = \ell - 2.\]
If  $\ell$ is odd, then we have  $\beta \geq (\ell-3)/2$ and  it suffices to take
\[ |\mathcal{G}| \geq 2 \left( \ell - \frac{\ell-3}{2} - 2\right) = \ell - 1.\]

Thus we have proven the following.
\begin{lemma}
\label{gap.lem}
If $\ell$ is even, then the gap $\mathcal{G}$ can have any  length $\geq \ell -2$, and  
if $\ell$ is odd, then the gap $\mathcal{G}$ can have any  length $\geq \ell -1$.
\end{lemma}

\begin{figure}[tb]
\begin{center}
\begin{tabular}{l}
Input: an \STS$(v)$ and an integer $\ell \geq 3$\\
$L \leftarrow \binom{\ell-1}{2}$\\
{\bf for} $i \leftarrow 1$ {\bf to} $L$ {\bf do}\\
\quad\quad Fill in the values in $\mathcal{B}_{i}^L, \mathcal{C}_{i}$ and  $\mathcal{B}_{i}^R$ using the greedy strategy\\
\quad\quad Compute the set $Y = \{y_1, \dots , y_m\}$.\\
\quad\quad Place the elements in $Y$ into $\mathcal{O}_{i}$ as described in Figure \ref{overflow.fig}.\\
\quad\quad Fill in the rest of $\mathcal{O}_{i}$ using the ``modified greedy strategy.''\\
 Fill in the points in $\mathcal{G}$ using the greedy strategy.\\
Compute $X \setminus \{ x_1, x_2, \dots , x_{v-L}\} = \{ \alpha_1, \dots , \alpha_L \}$.\\
{\bf for} $i \leftarrow 1$ {\bf to} $L$ {\bf do}\\
\quad\quad $x_{v-L+i} \leftarrow \alpha_i$\\
\quad\quad If necessary, swap $x_{v-L+i}$ with one of $\alpha_{i+1}, \dots , \alpha_L$ or a point from 
$\mathcal{C}_i$.\\
{\bf Return} $(\pi = [x_1\; x_2 \;  \cdots \; x_v])$.
\end{tabular}
\end{center}
\caption{Algorithm to find an $\ell$-good sequencing for an \STS$(v)$, $(X,\B)$}
\label{alg}
\end{figure}

\subsection{The algorithm}

Finally, the last $L$ points may be swapped (as described above) in order to ensure that we have an $\ell$-good sequencing.
Putting all the pieces together, we obtain the algorithm presented in Figure \ref{alg}.
The following lemma establishes the correctness of the algorithm.

\begin{lemma}
There is no block contained in $\ell$ consecutive points of $\mathcal{S}_i$ after a swap.
\end{lemma}

\begin{proof}
Suppose a block $B$ is  contained in $\ell$ consecutive points of $\mathcal{S}_i$ after a swap.
Clearly, $B$ must contain $\alpha_i$, which is the point that was ``swapped in.'' 
 Suppose that $\{z_{j_1}, z_{j_2}, \alpha_i\}$ is such a block, 
 where $j_1 < j_2$. Then $(j_1,j_2) \in \mathcal{J}_i$
 and $\alpha_i = \third (j_1,j_2)$.  However, it must be the case that $\third (j_1,j_2) \in Y$, in which case it occurs in the overflow; or $\third (j_1,j_2) \in \mathcal{T}_i \setminus \{\alpha_i\}$. In each case,
 $\alpha_i \neq \third (j_1,j_2)$, so we have a contradiction.
\end{proof}

\subsection{Analysis}

In this section, we prove our general existence result. 
Recall that we have various components in our sequencing:
\begin{itemize}
\item $L$ segments, 
$\mathcal{S}_i$ ($1 \leq i \leq L$), each consisting of 
\begin{itemize}
\item for $i \geq 2$, a left buffer of size $\ell-1$,
\item a core of size $i$, 
\item a right buffer of size $\ell-1$, and 
\item an overflow, whose size is given in Corollary \ref{overflow.cor}.
\end{itemize}
\item the gap $\mathcal{G}$ of size $\geq \ell-1$, and 
\item the final $L$ elements.
\end{itemize}
Therefore a sequencing of an \STS$(v)$ will exist if $v$ is at least as big as the sum of the lengths of all the components enumerated above:
\begin{eqnarray*}
v &\geq& \sum_{i=1}^L ( |\mathcal{B}_i^L| + |\mathcal{C}_i| + |\mathcal{B}_i^R| + |\mathcal{O}_i|) 
+ |\mathcal{G}| + L \\
& = & |\mathcal{C}_1| + |\mathcal{B}_1^R| + |\mathcal{O}_1| 
+ \sum_{i=2}^L ( |\mathcal{B}_i^L| + |\mathcal{C}_i| + |\mathcal{B}_i^R| + |\mathcal{O}_i|) + |\mathcal{G}| + L \\
& = & 
1 + \ell-1  + |\mathcal{O}_1| + 
\sum_{i=2}^L ( \ell-1 + i + \ell-1 + |\mathcal{O}_i|)
+ \ell-1 + L \\
& = & 
2\ell - 1 + L   + |\mathcal{O}_1| + 
\sum_{i=2}^L ( 2\ell-2 + i  + |\mathcal{O}_i|) \\
& = & 
2\ell - 1 + \binom{\ell-1}{2}   + \frac{\ell^3 - \ell^2 +2\ell-4}{4}\\
& & {} +
\sum_{i=2}^{\binom{\ell-1}{2}} \left( 2\ell-2 + i  + \frac{i(\ell^2 - \ell)}{2} + 
\frac{3 \ell^3 - 7 \ell^2 +6\ell - 4 }{4}\right).
\end{eqnarray*}

After some simplification, the following is obtained.

\begin{theorem}
\label{general.thm}
An \STS$(v)$ with
\begin{equation}
\label{general-bound}
v \geq \frac{(\ell-1)(\ell^5 - 9\ell^3 + 20\ell^2 - 36\ell + 16)}{16}
\end{equation}
has an $\ell$-good sequencing.
\end{theorem}

Here is a simpler bound that follows from Theorem \ref{general.thm}.

\begin{corollary}
An \STS$(v)$ with
$v \geq \ell^6/16$ has an $\ell$-good sequencing.
\end{corollary}

\begin{proof}
Consider the polynomial
\[
9\ell^3-20\ell^2+36\ell-16.
\]
This polynomial has a single root at $\ell \approx 0.58421$. Since $\ell \geq 3$, we know that $9\ell^3-20\ell^2+36\ell-16 > 0$, from which it follows that
\[
\ell^5 - 9\ell^3 + 20\ell^2 - 36\ell + 16 < \ell^5.
\]
Clearly, $\ell - 1 < \ell$, so 
\[(\ell-1)(\ell^5 - 9\ell^3 + 20\ell^2 - 36\ell + 16) < \ell^6\] for $\ell \geq 3$.
Hence,  (\ref{general-bound}) holds, and the result follows from Theorem \ref{general.thm}.
\end{proof}

For small values of $\ell$, we obtain the explicit bounds on $n(\ell)$ given in Table \ref{bounds.tab}.
We obtain slightly stronger bounds than Theorem \ref{general.thm} by using a gap of size $\ell-2$ when feasible 
(see Lemma \ref{gap.lem}) and a more precise bound on the size of 
the overflow $\mathcal{O}_i$ when $\ell$ is odd and $|\mathcal{J}_i|$ is even, 
as described in the proof of Lemma \ref{overflow.lem}.

Note that the upper bound on $n(4)$ is not as good as the one proven in \cite{KS2}. Of course, the result
from \cite{KS2} is obtained from an algorithm that was specially designed for the case $\ell=4$.

\begin{table}
\caption{Upper bounds on $n(\ell)$}
\label{bounds.tab}
\begin{center}
\begin{tabular}{r | r}
\multicolumn{1}{c|}{$\ell$} & \multicolumn{1}{c}{$n(\ell) \leq $} \\
\hline
4 & 119 \\
5 & 556 \\
6 & 1984 \\
7 & 5270 \\
8 & 12760 \\
9 & 26400 \\
10 & 52118
\end{tabular}
\end{center}
\end{table}

\section{Discussion and Conclusion}

Our algorithm is based on ideas from \cite{KS2}, where an algorithm to find a $4$-good sequencing of an 
\STS$(v)$ was developed. The algorithm from \cite{KS2} also employed the ``greedy strategy,'' ``modified greedy strategy'' and  an overflow (although the latter term was not used in \cite{KS2}) in much the same way as the present algorithm. In \cite{KS2}, only a single overflow and swap was needed, which meant that the algorithm would work for smaller values of $v$ than the general algorithm we describe in this paper. However, the approach in \cite{KS2} did not seem to generalize well to larger values of $\ell$, so the algorithm we have presented here employs a series of (up to) $L$ swaps that take place in disjoint intervals. This permits the development of an algorithm for arbitrary values of $\ell$. 

It would of course be of interest to obtain more accurate upper and lower bounds on $\ell$ (as a function of $v$)
for the existence of $\ell$-good sequencings of \STS$(v)$. Phrased in terms of asymptotic complexity, 
our necessary condition is that
$\ell$ is $O(v)$, while the sufficient condition proven in this paper is that $\ell$ is $\Omega(v^{1/6})$. Closing this gap is an interesting open problem.

\end{document}